\documentclass{article}
\usepackage[utf8]{inputenc}
\usepackage{amsmath, amssymb}
\usepackage{amsthm}
\usepackage{color}
\usepackage{soul}
\usepackage{cite, hyperref}

\usepackage{pgf,tikz}
\usetikzlibrary{arrows}
\usetikzlibrary{arrows.meta} 

\newcommand\blfootnote[1]{%
  \begingroup
  \renewcommand\thefootnote{}\footnote{#1}%
  \addtocounter{footnote}{-1}%
  \endgroup
}

\newtheorem{theorem}{Theorem}[section]

\newtheorem{prop}[theorem]{Proposition}
\newtheorem{sejt}[theorem]{Conjecture}
\newtheorem{lemma}[theorem]{Lemma}

\DeclareMathOperator{\aut}{Aut}
\newcommand{\D}{\mathcal{D}}
\newcommand{\C}{\mathcal{C}}

\title{On the Automorphism Group of the Substructure Ordering of Finite Directed Graphs}
\author{Fanni K. Nedényi \ and \ Ádám Kunos}
\date{}

\tikzset{My Style/.style={circle,black, draw=black, fill=black, minimum size=0.1cm,inner sep=2pt}}

\begin{document}

\maketitle

\begin{abstract}
    We investigate the automorphism group of the substructure ordering of finite directed graphs. 
    The second author conjectured that it is isomorphic to the 768-element group $(\mathbb{Z}_2^4 \times S_4)\rtimes_{\alpha} \mathbb{Z}_2$. 
    Though unable to prove it, we solidify this conjecture by showing that the automorphism group behaves as expected by the conjecture on the first few levels of the poset in question.
    With the use of computer calculation we analyze the first four levels holding 3160 directed graphs.
\end{abstract}

\section{Introduction}

\blfootnote{\it 2020 Mathematics Subject Classifications: 06A06, 06A07, 68V05.}
\blfootnote{\it Key words and phrases: directed graph, digraph, partially ordered set, poset, automorphism group.}
\blfootnote{
Fanni K. Nedényi: {\it nfanni@math.u-szeged.hu},
ELKH-SZTE Analysis and Applications Research Group,
   Bolyai Institute, University of Szeged,
   Aradi v\'ertan\'uk tere 1, H--6720 Szeged, Hungary.}
\blfootnote{Ádám Kunos: {\it akunos@math.u-szeged.hu},
Bolyai Institute, University of Szeged, Aradi v\'ertan\'uk tere 1, H-6720 Szeged, Hungary.}
\blfootnote{The research was partially supported by the project TKP2021-NVA-09 from the National Research, Development and Innovation Fund of the Ministry of Innovation and Technology, Hungary.} 
\blfootnote{The second author was partially supported by the National Research, Development and Innovation Office (Hungary), Grants K128042 and K138892.}

In this paper, we are investigating the automorphism group of a concrete, infinite partially ordered set. 
As usual, a {\it directed graph}, or {\it digraph} for short, is a nonempty set (i. e. the set of vertices) equipped with a binary relation (i. e. the set of edges).
For a digraph $G$ let $V(G)$ denote its set of vertices.
Let $\mathcal{D}$ denote the set of the isomorphism types of finite directed graphs.
For $G, H \in \mathcal{D}$, let $G\sqsubseteq H$ stand for the fact that $G$ is (isomorphic to) a spanned subgraph of $H$. 
Then $(\mathcal{D}; \sqsubseteq)$ is clearly a partially ordered set, shorty {\it poset}. 
Figure \ref{41511} shows the first two levels of the poset $(\mathcal{D}, \sqsubseteq)$.

\begin{figure}
\begin{center}
\begin{tikzpicture}[line cap=round,line join=round,>=Latex, x=.49cm,y=.7cm]
\clip(-0.01,.95) rectangle (30,6.8);

\draw (0,5)-- (2,5);
\draw (2,5)-- (2,6);
\draw (2,6)-- (0,6);
\draw (0,6)-- (0,5);
\fill [color=black] (0.5,5.5) circle (1.5pt);
\fill [color=black] (1.5,5.5) circle (1.5pt);
\draw (1,5)-- (11.5, 2);
\draw[color=black] (1,6.5) node {$E$};

\draw (2.5,5)-- (4.5,5);
\draw (4.5,5)-- (4.5,6);
\draw (4.5,6)-- (2.5,6);
\draw (2.5,6)-- (2.5,5);
\fill [color=black] (3,5.5) circle (1.5pt);
\fill [color=black] (4,5.5) circle (1.5pt);
\draw [->] (3,5.5) -- (4,5.5);
\draw (3.5,5)-- (11.5, 2);
\draw[color=black] (3.5,6.5) node {$P$};

\draw (5,5)-- (7,5);
\draw (7,5)-- (7,6);
\draw (7,6)-- (5,6);
\draw (5,6)-- (5,5);
\fill [color=black] (5.5,5.5) circle (1.5pt);
\fill [color=black] (6.5,5.5) circle (1.5pt);
\draw [->] (5.5,5.5) ..controls(6, 5.7)..  (6.5,5.5);
\draw [->] (6.5,5.5) ..controls(6, 5.3)..  (5.5,5.5);
\draw (6,5)-- (11.5, 2);
\draw[color=black] (6,6.5) node {$E'$};

\draw (7.5,5)-- (9.5,5);
\draw (9.5,5)-- (9.5,6);
\draw (9.5,6)-- (7.5,6);
\draw (7.5,6)-- (7.5,5);
\fill [color=black] (8,5.5) circle (1.5pt);
\draw[rotate around={-60:(8,5.5)}] [->] (8,5.5) arc (360:10:3pt);
\fill [color=black] (9,5.5) circle (1.5pt);
\draw (8.5,5)-- (11.5, 2);
\draw (8.5,5)-- (13, 2);
\draw[color=black] (8.5,6.5) node {$A$};

\draw (10,5)-- (12,5);
\draw (12,5)-- (12,6);
\draw (12,6)-- (10,6);
\draw (10,6)-- (10,5);
\fill [color=black] (10.5,5.5) circle (1.5pt);
\draw[rotate around={-60:(10.5,5.5)}] [->] (10.5,5.5) arc (360:10:3pt);
\fill [color=black] (11.5,5.5) circle (1.5pt);
\draw [->] (10.5,5.5) -- (11.5,5.5);
\draw (11,5)-- (11.5, 2);
\draw (11,5)-- (13, 2);
\draw[color=black] (11,6.5) node {$B$};

\draw (12.5,5)-- (14.5,5);
\draw (14.5,5)-- (14.5,6);
\draw (14.5,6)-- (12.5,6);
\draw (12.5,6)-- (12.5,5);
\fill [color=black] (13,5.5) circle (1.5pt);
\draw[rotate around={-45:(13,5.5)}] [->] (13,5.5) arc (360:10:3pt);
\fill [color=black] (14,5.5) circle (1.5pt);
\draw [->] (14,5.5) -- (13,5.5);
\draw (13.5,5)-- (11.5, 2);
\draw (13.5,5)-- (13, 2);
\draw[color=black] (13.5,6.5) node {$C$};

\draw (15,5)-- (17,5);
\draw (17,5)-- (17,6);
\draw (17,6)-- (15,6);
\draw (15,6)-- (15,5);
\fill [color=black] (15.5,5.5) circle (1.5pt);
\draw[rotate around={-40:(15.5,5.5)}] [->] (15.5,5.5) arc (360:10:3pt);
\fill [color=black] (16.5,5.5) circle (1.5pt);
\draw [->] (15.5,5.5) ..controls(16, 5.7)..  (16.5,5.5);
\draw [->] (16.5,5.5) ..controls(16, 5.3)..  (15.5,5.5);
\draw (16,5)-- (11.5, 2);
\draw (16,5)-- (13, 2);
\draw[color=black] (16,6.5) node {$D$};

\draw (17.5,5)-- (19.5,5);
\draw (19.5,5)-- (19.5,6);
\draw (19.5,6)-- (17.5,6);
\draw (17.5,6)-- (17.5,5);
\fill [color=black] (18,5.5) circle (1.5pt);
\draw[rotate around={-90:(18,5.5)}] [->] (18,5.5) arc (360:10:3pt);
\fill [color=black] (19,5.5) circle (1.5pt);
\draw[rotate around={-90:(19,5.5)}] [->] (19,5.5) arc (360:10:3pt);
\draw (18.5,5)-- (13, 2);
\draw[color=black] (18.5,6.5) node {$L$};

\draw (20,5)-- (22,5);
\draw (22,5)-- (22,6);
\draw (22,6)-- (20,6);
\draw (20,6)-- (20,5);
\fill [color=black] (20.5,5.5) circle (1.5pt);
\draw[rotate around={-40:(20.5,5.5)}] [->] (20.5,5.5) arc (360:10:3pt);
\fill [color=black] (21.5,5.5) circle (1.5pt);
\draw[rotate around={-90:(21.5,5.5)}] [->] (21.5,5.5) arc (360:10:3pt);
\draw [->] (20.5,5.5) -- (21.5,5.5);
\draw (21,5)-- (13, 2);
\draw[color=black] (21,6.5) node {$Q$};

\draw (22.5,5)-- (24.5,5);
\draw (24.5,5)-- (24.5,6);
\draw (24.5,6)-- (22.5,6);
\draw (22.5,6)-- (22.5,5);
\fill [color=black] (23,5.5) circle (1.5pt);
\draw[rotate around={-40:(23,5.5)}] [->] (23,5.5) arc (360:10:3pt);
\fill [color=black] (24,5.5) circle (1.5pt);
\draw[rotate around={-90:(24,5.5)}] [->] (24,5.5) arc (360:10:3pt);
\draw [->] (23,5.5) ..controls(23.5, 5.7)..  (24,5.5);
\draw [->] (24,5.5) ..controls(23.5, 5.3)..  (23,5.5);
\draw (23.5,5)-- (13, 2);
\draw[color=black] (23.5,6.5) node {$L'$};

\draw (12,1)-- (12,2);
\draw (12,2)-- (11,2);
\draw (11,2)-- (11,1);
\draw (11,1)-- (12,1);
\fill [color=black] (11.5,1.5) circle (1.5pt);

\draw (13.5,1)-- (13.5,2);
\draw (13.5,2)-- (12.5,2);
\draw (12.5,2)-- (12.5,1);
\draw (12.5,1)-- (13.5,1);
\fill [color=black] (13,1.3) circle (1.5pt);
\draw[rotate around={-90:(13,1.3)}] [->] (13,1.3) arc (360:10:4pt);

\end{tikzpicture}
\end{center}
\caption{The initial segment of the Hasse diagram of the substructure ordering of digraphs, that is $(\mathcal{D}; \sqsubseteq)$. }
\label{41511}
\end{figure}
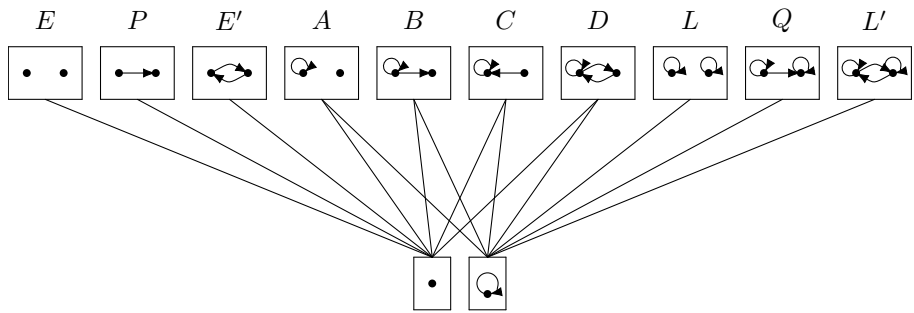

In this paper, we are investigating the automorphism group $\aut\mathcal{D}$ of the poset $\mathcal{D}$.
In \cite{Ksubstructure}, the second author showed 768 different automorphisms, and conjectured there is no more (see Conjecture \ref{conjecture} in the present paper for details).
The conjecture was quite baseless in the sense that it only conveyed the author's feelings.
In this paper, we attack the conjecture using some computer calculation, and, though unable to prove it, we give it a strong basis---proving the automorphism group acts on the first three levels of $\mathcal{D}$ in accordance with the conjecture.

This conjecture is interesting because it breaks a pattern that has been established in a series of papers. 
In 2009--2010 Jaroslav Je\v{z}ek and Ralph McKenzie published a series of papers \cite{Jezek2009_1, Jezek2010, Jezek2009_3, Jezek2009_4} in which they examined similar orderings, determined the automorphism groups of them (among other things). 
Instead of directed graphs, they considered finite semilattices \cite{Jezek2009_1}, ordered sets \cite{Jezek2010}, distributive lattices \cite{Jezek2009_3} and lattices \cite{Jezek2009_4}---but the main concept was identical.
A little later, Alexander Wires and the second author continued the series with three papers \cite{Wires2016, Kunos2015, Kunos2019}, investigating simple graphs and directed graphs, respectively. 
Note that for directed graphs, in the papers \cite{Kunos2015, Kunos2019}, the second author considered an ordering that is different from the one being investigated in the current paper. 
All these papers displayed a pattern: the corresponding automorphism groups turned out to be isomorphic to either the trivial or two element group.
The phenomenon that nontrivial automorphisms enter the picture had been unprecedented before \cite{Ksubstructure}.
For this reason, the usual line of arguments fails, leaving the automorphism group unsettled.
Something new is required.


A posed is called {\it graded} if it has levels in it, the natural way.
For a graded poset $\mathcal{P}$, let $\mathcal{P}_n$ denote the subposet of the first $n$ levels. 
It is clear that $\mathcal{P}_n$ is invariant under $\aut \mathcal{P}$.
Let $\varphi|_n$ denote the restriction of $\varphi\in \aut \mathcal{P}$ to $\mathcal{P}_n$.
Furthermore, let $\aut_n\mathcal{P}$ denote $\{\varphi|_n \in \aut\mathcal{P}_n : \varphi \in \aut \mathcal{P}\}$, which is clearly a subgroup of $\aut\mathcal{P}_n$.

Observe that our poset $(\D; \sqsubseteq)$ is graded and $\D_n=\{G\in\D : |V(G)| \leq n\}$.
In \cite{Ksubstructure}, though not stated in this way, it was shown that 

\begin{prop}\label{43426869}
For $\varphi, \varphi' \in \aut\D$, $\varphi|_{12}=\varphi'|_{12}$ implies $\varphi=\varphi'$.
\end{prop}

Unfortunately, it is not straightforward to extract this proposition from \cite{Ksubstructure}.
Hence an explanation is due, but as it would take us off our main course here, this explanation is put at the end of this section, in a standalone subsection.

Proposition \ref{43426869} yields that it is enough to determine $\aut_{12}\D$.
This is what makes our problem finite---at least ostensibly.
Though the determination of $\aut_{12}\D$ is not a finite problem, the determination of $\aut\D_{12}$ is.
Let $\mathcal{C}$ denote the group of our conjecture. 
(For now, we don't need to know its exact structure, $\mathcal{C}$ is just a notation. 
But if interested, see Conjecture \ref{conjecture}).
At first sight, it is natural to expect $\aut\D_n$ to be isomorphic to $\C$ for a big enough $n$.
One might expect $n=12$ to be big enough---meaning $\aut\D_{12}\cong \C$.
This, together with $\aut_{12}\D\leq \aut\D_{12}$, would yield our conjecture as in Section \ref{automorphisms} we show a subgroup of the automorphism group isomorphic to $\C$ that does not collapse when restricted to the first three levels (not to mention the first twelve). 
However, the unfortunate fact is that we cannot expect $\aut\D_n$ to be isomorphic to $\C$ at all.
The reason for this is actually the statement of Lemma \ref{extension}. 
After that lemma, as an application, we show that $|\aut\D_3|\approx 1.67\cdot 10^{13}$ which is quite far from the desired 768.
What can we do then?
Let us consider the following chain of subgroups of $\aut\D_{12}$:
$$ \aut\D_{12} \geq \aut_{12}\D_{13} \geq \aut_{12}\D_{14} \geq\dots $$
All elements of this chain contain the sought $\aut_{12}\D$ as a subgroup and it is clearly a finite problem to determine each of them. 
What's more is that we can actually expect this chain to stabilize at the yearned $\aut_{12}\D$.
What we prove in this paper is that (even though $|\aut\D_3|\approx 1.67\cdot 10^{13}$) the chain
$$ \aut\D_{3} \geq \aut_{3}\D_{4} \geq \aut_{3}\D_{5} \geq \dots $$
stabilizes already at its second element as the following theorem holds.

\begin{theorem}\label{982thss}
$\aut_3\D =\aut_3\mathcal{D}_4 \cong \C$.
\end{theorem}

This theorem solidifies the conjecture to a great extent, even though far from being a proof of it.

\subsection{On Proposition \ref{43426869}}

The paper \cite{Ksubstructure} deals mainly with first-order definability (in the poset $(\mathcal{D}; \sqsubseteq)$), just as its predecessors \cite{Jezek2009_1, Jezek2010, Jezek2009_3, Jezek2009_4, Wires2016, Kunos2015, Kunos2019} did. 
It was always the definability results that one could derive the automorphism groups from---as easy corollaries.
The connection is not so one-way though.
Automorphisms obstruct definability results in a quite natural way.
For example, for an automorphism $\varphi\in \mathcal{D}$, the digraphs $G$ and $\varphi(G)$ are not distinguishable by first-order formulas as those can only capture structural properties of the poset $(\mathcal{D}; \sqsubseteq)$, and, there, $G$ and $\varphi(G)$ clearly share all those (considered just as elements of $(\mathcal{D}; \sqsubseteq)$).

One definability result, namely Corollary 2 of \cite{Ksubstructure} says that after adding a finite number of constants from $\mathcal{D}$ to the first order language of posets, every single digraph becomes first-order definable.
For the automorphism group, this means that after fixing those constants, there remains no other automorphism besides the trivial one.
If we look at the proof, the constants used are from the first 12 levels of $(\mathcal{D}; \sqsubseteq)$.
This is where Proposition \ref{43426869} comes from.

We have to remark here that, at the very end of the paper \cite{Ksubstructure}, there is a discussion on the constants used and there is an error in it.
It is said there that our constants come from the first 9 levels, where the number 9 comes from a formula, $(p+1)q$, with the substitution $p=2$ and $q=3$.
The formula is actually wrong, it should be $(p+1)(q+1)$, and therefore the correct number is $(2+1)(3+1)=12$.
The formula itself is based on a very basic combinatorial argument that one can verify easily (maybe that is the reason it was not given enough attention and got overlooked by the second author at the time).

\section{768 Distinct Automorphisms} \label{automorphisms}

Automorphisms map digraphs to digraphs in $\mathcal{D}$.
To define an automorphism $\varphi$, we need to tell how to get $\varphi(G)$ from $G$.
All the automorphisms, that we know of at the moment, share a particular characteristic.
They are all, say, {\it local} in the following sense.
Roughly speaking, to get $\varphi(G)$ from $G$, one only needs to consider and modify the at most two element substructures of $G$ according to some given rule.

To make this clearer, we give an example. 
Let $\varphi(G)$ be the digraph that we get from $G$ such that we change the direction of the edges on those two element substructures of $G$ that have loops on both vertices.
It is easy to see that this defines an automorphism, indeed.
Perhaps, one would quickly discover the automorphism that gets $\varphi(G)$ by reversing all edges of $G$, but this is different.  
In this example, the modification of $G$ happens only locally, namely on 2-element substructures.
All the automorphisms, that we know of, share this property.

Now, we define some of our automorphisms, $\varphi_i$. 
We tell how to get $\varphi_i(G)$ from $G$. 
One of the most trivial automorphisms is
\begin{itemize}
\item $\varphi_1$: where there is a loop, clear it, and vice versa: to the vertices with no loop, insert one.
\end{itemize}
Observe that this automorphism operates with the 1-element substructures. 
Now we start to make use of the labels of Fig. \ref{41511}.
\begin{itemize}
\item $\varphi_2$: change the substructures (isomorphic to) $E$ to $E'$ and vice versa.
\item $\varphi_3$: change the substructures (isomorphic to) $L$ to $L'$ and vice versa.
\item $\varphi_4$: reverse the edges in the substructures (isomorphic to) $P$.
\item $\varphi_5$: reverse the edges in the substructures (isomorphic to) $Q$.
\end{itemize}
Let $S_4$ denote the symmetric group over the four-element set of digraphs $\{A,B,C,D\}$ (see Fig. \ref{41511}), and $\pi \in S_4$.
We define
\begin{itemize}
\item $\varphi_\pi$: change the substructures (isomorphic to) $X \in \{A,B,C,D\}$ to $\pi(X)$ (such that the loops remain in place).
\end{itemize}

Let $I$ denote the set of possible indices of our $\varphi$s, that is
$$I=\{1,\dots,  5\} \cup \{\pi \in S_4\}.$$
Though not hard to see, but to be precise, we need the following lemma.

\begin{lemma} Every $\varphi_i$ $(i\in I)$, defined above, is in fact an automorphism. 
\end{lemma}

\begin{proof}
For the sake of simplicity, let $\varphi=\varphi_i$ for this proof. 
It is easy to see that $\varphi$ is a bijection and $\varphi^{24}=\text{id}$ holds (24 being the least common multiple of their orders).
Therefore, we claim that it is enough to prove 
\begin{equation}\label{uuzr53}
G\sqsubseteq G' \;\Rightarrow\; \varphi(G) \sqsubseteq \varphi(G').
\end{equation}
This is because from $\varphi(G) \sqsubseteq \varphi(G')$, applying \eqref{uuzr53} 23 times, we obtain
$$\varphi(G) \sqsubseteq \varphi(G') \Rightarrow \varphi^2(G) \sqsubseteq \varphi^2(G') \Rightarrow \dots  \Rightarrow G=\varphi^{24}(G) \sqsubseteq \varphi^{24}(G')=G',$$
which is exactly the reverse direction of \eqref{uuzr53}. \\
Finally, to see that \eqref{uuzr53} holds, we show that an embedding $\psi: G \to G'$ demonstrating the fact $G\sqsubseteq G'$ demonstrates $\varphi(G) \sqsubseteq \varphi(G')$ as well. 
This is because all of our $\varphi_i$-s operate on two-element substructures of $G$, and for two-element digraphs it is easy to see that \eqref{uuzr53} holds.
\end{proof}


Let $\langle \rangle$ stand for subgroup generation.
Let $\mathcal{C}=\langle\varphi_i : i\in I\rangle$.
Though, we conjecture that $\C=\aut\D$, now (without computer calculations), we only prove the following.

\begin{prop} The subgroup $\C\leq \aut\D$  is isomorphic to $(\mathbb{Z}_2^4 \times S_4)\rtimes_{\alpha} \mathbb{Z}_2$ (with the $\alpha$ defined around the equation (\ref{uzzrhuiwre})).
\end{prop}

\begin{proof}
First, we prove that $\mathcal{C}$ splits into the internal semidirect product
\begin{equation}\label{xxcv7522236}
\mathcal{C}=\langle\varphi_i : i\in I\setminus\{1\}\rangle \rtimes \langle \varphi_1\rangle.
\end{equation}
Let $S=\langle\varphi_i : i\in I\setminus\{1\}\rangle$.
To prove \eqref{xxcv7522236}, we have to check that \\
1) $S$ is a normal subgroup of $\C$, \\
2) $S \cap \langle \varphi_1\rangle = \{\text{id}\}$, and \\
3) $S \langle \varphi_1\rangle = \C$.

To show 1), namely $\varphi_1 S \varphi_1^{-1}\subseteq S$, first observe $\varphi_1^{-1}=\varphi_1$.
Therefore what we need is $\varphi_1 \varphi_i \varphi_1 \in S$ for $s\in S$ and $i\in I\setminus\{1\}$.
The following are clear from the definition of $\varphi_i$:
\begin{equation} \label{ertr876jd}
\varphi_1 \varphi_2\varphi_1=\varphi_3, \;\;
\varphi_1 \varphi_3\varphi_1=\varphi_2, 
\;\; \text{ and analogously} \;\;
\varphi_1 \varphi_4\varphi_1=\varphi_5, \;\;
\varphi_1 \varphi_5\varphi_1=\varphi_4
\end{equation} 
Furthermore, as usual, let $(BC)$ be the cycle notation of the permutation of $S_4$ that takes $B$ to $C$ and vice versa.
If $\pi \in S_4$, and $\pi'=(BC)\pi(BC)$, then it is easy to check that 
\begin{equation}\label{kkfshdfe}    
\varphi_1 \varphi_\pi \varphi_1=\varphi_{\pi'}
\end{equation}
We are done with the proof of 1).

To prove 2), observe that the generators of $S$ do not touch loops, that is the loops in $\varphi_i(G)$ ($i\in I\setminus\{1\}$) are in the `same position' as in $G$. 
Hence there is no element of $S$ touching loops, which $\varphi_1$ clearly does. 

Let's turn to 3). Rearranging the equations \eqref{ertr876jd} and \eqref{kkfshdfe}, like $\varphi_1 \varphi_i=\varphi_i \varphi_1$, we can move all $\varphi_1$s to the right side of any expression in $S$, and that is all we need to do.

Furthermore, we state that S is an internal direct product
$$S=\langle\varphi_2 \rangle \times \langle\varphi_3\rangle \times \langle\varphi_4\rangle \times \langle\varphi_5\rangle \times \langle\varphi_\pi : \pi \in S_4\rangle.$$ 
Note that here, at the last factor, the subgroup generation is just a technicality as, clearly, the $\varphi_\pi$-s constitute a subgroup themselves.
Let us show
$$S=\langle\varphi_i : i\in \{2,3,4,5\}\rangle \times \langle\varphi_\pi : \pi \in S_4\rangle$$
first. 
This comes from the fact that the automorphisms of the first and the second factors touch different two-element substructures of the digraphs, hence their intersection is trivial and their elements commute as well.
The fact 
$$\langle\varphi_i : i\in \{2,3,4,5\}\rangle = \langle\varphi_2 \rangle \times \langle\varphi_3\rangle \times \langle\varphi_4\rangle \times \langle\varphi_5\rangle$$
comes from a very similar argument. \\
What we've proven is that $\mathcal{C}$ is isomorphic to
$$(\mathbb{Z}_2^4 \times S_4)\rtimes_{\alpha} \mathbb{Z}_2,$$
where $S_4$, again, denotes the symmetric group over the set $\{A,B,C,D\}$, and $\alpha$ is the following. 
Obviously, $\alpha(0)=\text{id}\in  \mathrm{Aut}(\mathbb{Z}_2^4 \times S_4)$.
To define $\alpha(1)$, let $p,q,r,s \in \{0,1\}$ and $\pi \in S_4$.
Then
\begin{equation} \label{uzzrhuiwre}
\alpha(1): (p,q,r,s,\pi) \mapsto (q,p,s,r, (BC)\pi(BC)).
\end{equation}
\end{proof}

The group $\C$ of our conjecture has 768 elements.
Even though we cannot prove that there are no more automorphisms beyond the ones in $\C$, we conjecture so.

\begin{sejt}\label{conjecture} $\aut\D \cong \C$, i. e. the automorphism group of the partial order $(\mathcal{D}; \sqsubseteq)$ is isomorphic to $(\mathbb{Z}_2^4 \times S_4)\rtimes_{\alpha} \mathbb{Z}_2$, with the $\alpha$ defined above (around (\ref{uzzrhuiwre})).
\end{sejt}

\section{The Extension Lemma}

Let $G\prec H$ mean that the digraph $H$ covers $G$ in the poset $(\mathcal{D}; \sqsubseteq)$.
Let $l(G)$ denote the set of lower covers $G$, that is 
$$l(G)=\{H\in \D: H\prec G\}.$$
When trying to prove our conjecture, probably the most natural thing to do would be to prove it for the bottom of the poset $\D$ (that is actually what we do in Theorem \ref{982thss}) and show that from a certain point upwards in $\D$, $l(G)$ is always unique. 
If this was the case, one can see that fixing an automorphism at the bottom would fix it everywhere and we would be done.
Is it true that for big enough $G, H\in \D$, $l(G)=l(H)$ implies $G=H$?
One might recognize that, in fact, this is some sort of reconstruction problem for directed graphs. 
Unfortunately reconstruction conjectures are hard and they tend to turn out false even with much stronger assumptions.
We will later see, with our computer calculations, that for the first four levels of our $\D$, such a reconstruction statement is far from being true.

There are 104 digraphs on the third level. To check $104!$ permutations is clearly out of reach, so brute force calculation is out of question even for automorphisms of the first three levels.
This makes us resort to more refined approaches.

Let's say we want to decide for a concrete $\varphi: \D_n\to \D_n$ if it is a member of $\aut\D_n$.
What can we do?
It is clear that automorphisms only move elements inside their levels and therefore it is enough to check if $\varphi$ is a bijection, leaves everyone on its level, and preserves the covering relation, that is for all $G, H \in \D_n$
\begin{equation}\label{sjfhdj8}
G\prec H\; \Longleftrightarrow \; \varphi(G)\prec\varphi(H)
\end{equation}
holds.
The problem with this is that it takes much calculation because there are many covering pairs.

We've already encountered two computational problems: the difficulty of having to check too many $\varphi$s and the time-consuming nature of checking even one.
It is time to introduce our medicine for these two problems, the so-called Extension Lemma. 
It describes whether an automorphism $\varphi\in \aut \D_n$ extends to an automorphism of $D_{n+1}$ in such a way that is easy to calculate by computers. 
Some technical notions have to be introduced first. \\
Let $s(G)$ denote the number of digraphs in $\D$ that share $l(G)$, that is
$$s(G)=|\{H\in \D: l(G)=l(H)\}|.$$
Let 
$$H_n=\{(l(G),s(G)):G\in \D, |V(G)|=n\}.$$
For a set $S\subseteq \D$ of digraphs, $\varphi(S)$ is just taken element-wise.
The pairs in $H_n$ are of the form $(S,k)$ where $S$ is a set of digraphs, and $k$ is a natural number. 
We define the action of $\varphi$ on such pairs as $\varphi(S,k)=(\varphi(S),k)$.
$\varphi(H_n)$ is, again, taken element-wise.

\begin{lemma}[Extension Lemma]\label{extension}
An automorphism $\varphi$ of $\D_n$ extends to $\D_{n+1}$ if and only if $\varphi(H_{n+1})=H_{n+1}$. 
Each of the extendible automorphisms has exactly
\begin{equation}\label{jnhztfgt}
    \prod_{(S,k)\in H_{n+1}}k!
\end{equation}
extensions.
\end{lemma}

Before its proof, let's see how we use this lemma in our calculations.
For example, let's determine $|\aut\D_3|$. 
Remember, there are 104 digraphs on the third level.
Again, the brute force calculation, which would check $104!\cdot 10!\cdot 2!$ maps (because there are 104, 10, and 2 digraphs on the first three levels), is clearly out of grasp, even for the strongest computer on earth.
Let us try to use the Extension Lemma.
By looking at Figure \ref{41511}, one can see that $|\aut\D_2|=2\cdot3!\cdot 4! \cdot 3!=1728$. 
It is fast for our programme to calculate $H_3$, it turns out to be a 74-element set.
And it is also fast to check the condition $\varphi(H_3)=H_3$ for all of the 1728 $\varphi$-s of $\aut\D_2$. 
It turns out that the condition is satisfied for 192 maps, meaning $|\aut_2\D_3|=192$.
Let's observe here that what we got means
$$\aut\D_2>\aut_2\D_3=\aut_2\D_4=\dots=\aut_2\D,$$
because the automorphisms unveiled in Section \ref{automorphisms} already show us 192 different restrictions to the first two levels. 
(This is explained in more detail in the proof of Theorem \ref{982thss}.)  
Our original question was to determine $|\aut\D_3|$.
The Extension Lemma answers that question as well for it says that each of the 192 automorphisms found has $\prod_{(S,k)\in H_{3}}k!$ automorphisms. 
Since we have $H_3$ at hand, it is easy to read out that the answer is
$$|\aut\D_3|=192 \cdot  (1!)^{52} \cdot (2!)^{18} \cdot (4!)^4\approx 1.67\cdot 10^{13}.$$
We now prove the Extension Lemma.

\begin{proof}
The fact that $\varphi \in \D_n$ means that we only have to verify \eqref{sjfhdj8} when $G$ is on the $n$-th level and $H$ is on the $(n+1)$-th.
Let us partition the digraphs on the $(n+1)$-th level according to the set of their lower covers, that is $G$ and $G'$ gets placed in the same partition if and only if $l(G)=l(G')$. 
Let us denote the partitions we get by $p_1, \dots, p_k$. 
Note that $(S,k)\in H_{n+1}$ if and only if there exists some $G\in p_i$ for which $l(G)=S$ and $|p_i|=k$. \\
First, we tackle the only if ($\Rightarrow$) direction of the first statement of the lemma. Suppose $\varphi$ extends to, say, $\varphi'\in \aut_n\D_{n+1}$. 
For $\varphi$ acts bijectively on the $n$-th level, it is injective on $H_{n+1}$, and, therefore, it is enough to show $\varphi(H_{n+1})\subseteq H_{n+1}$. 
Let us pick $(S,k)\in H_{n+1}$. As we saw before, there exists some $G\in p_i$ for which $l(G)=S$ and $|p_i|=k$.
It is clear that $\varphi(l(G))=l(\varphi'(G))$, and also $|\varphi'(p_i)|=|p_i|=k$. 
Therefore $\varphi(S,k)=(\varphi(S),k)\in H_{n+1}$, and that is what we wanted. \\
Now we show the if ($\Leftarrow$) direction, together with the formula \eqref{jnhztfgt}. Now we have a $\varphi\in \aut\D_n$ with the property $\varphi(H_{n+1})= H_{n+1}$, and we need to show that there are \eqref{jnhztfgt}-many extensions of it in $\aut_n\D_{n+1}$.
It is clear that $\varphi$ is bijective on $H_{n+1}$ and therefore induces a bijection $\psi$ on the set $\{p_1, \dots, p_k\}$ by  
$$\psi(p_i)=p_j \; \Longleftrightarrow \; (\forall G\in p_i)(\forall H\in p_j)( \varphi (l(G))=l(H)).$$
Note that the fact $\varphi(H_{n+1})= H_{n+1}$ makes this definition of $\psi$ correct.
It is easy to see that the extension $\varphi'$ of $\varphi$ is an automorphism if and only if
$$G\in p_i \Rightarrow \varphi'(G)\in \psi(p_i)$$
holds.
What this means is that there are as many extensions $\varphi'$ as permuting our digraphs inside the image of their own partition, $\psi(p_i)$, and that is exactly \eqref{jnhztfgt}.
\end{proof}

\section{The Proof of the Main Result}

In this section we prove our main result, Theorem \ref{982thss}.
Unfortunately, the Extension Lemma, in itself, is not enough to calculate the sought $|\aut_3\mathcal{D}_4|$ as it would require us to check $|\aut\D_3|\approx 1.67\cdot 10^{13} $ automorphisms, which is out of reach.
Hence, we must look for ways to exclude some inextendible elements of $\aut\D_3$. \\

\noindent{\it Proof of Theorem \ref{982thss}.}
Knowing the nature of our automorphisms it is clear that they all manifest on the first three levels---meaning that the automorphism group doesn't collapse when restricted to the first three levels. 
This fact already gives us $|\aut_3\D_4|\geq 768$, it only remains to be seen that $|\aut_3\D_4|\leq 768$.
In fact, observe that, even though the automorphisms $\varphi_4$ and $\varphi_5$ are both the identity restricted to the first two levels, the rest of the generators do manifest on the first two levels, independently of each other.
Hence the 192-element group
$\langle\varphi_i : i\in I\setminus\{4,5\}\rangle$ does not collapse on the first two levels.
This implies---using the usual orbit-stabilizer argument---that it is enough to show the following \\

\noindent {\it Claim.} There are at most four $\varphi \in \aut_3\D_4$ that fixes the first two levels. \\

\noindent That is what we are going to show in the rest of the proof. 
Now, we start to look for necessary conditions for $\varphi$ to extend. 
Let $G$ be a digraph on the third level.
$\varphi$ must preserve the lower covers of $G$ (for the first two levels are fixed) and the number of upper cowers of $G$ as well. 
Let's formalize this.
Let $u(G)$ and $l(G)$ denote the set of upper and lower covers of $G$ respectively.
Let us introduce two equivalence relations on the third level:
$$\alpha=\{(G,G'): l(G)=l(G'), \; |V(G)|=3\},$$
and 
$$\beta=\{(G,G'): |u(G)|=|u(G')|, \; |V(G)|=3\}.$$
Now, the observation we had above is
$(G, \varphi(G))\in \alpha \cap \beta$.
With computer calculation, it turns out that, beyond its one-element classes, the equivalence relation $\alpha \cap \beta$ has only 20 two-element classes.
This reduces the number of our possible $\varphi$-s to $2^{20}$, which is less but still too many to check brute-force.
Consequently, we have to keep reducing the number of possibilities.
As $\varphi_4$ and $\varphi_5$, two automorphisms of order two, leave the first two levels fixed, we expect to have two pairs, out of the 20 above, to move freely (even with the first two levels fixed). 
Reassuringly, this is what happens as one of the equivalence classes is $\{V,\varphi_4(V)\}$, and one other is $\{W, \varphi_5(W)\}$, where $V$ and $W$ are the two digraphs of Fig. \ref{fig:M}.

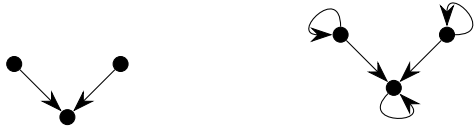
\begin{figure}[h]  
\centering 
\begin{tikzpicture}[main/.style = {draw, circle}] 
\node[My Style] (1) {}; 
\node[My Style] (2) [below right of=1] {};
\node[My Style] (3) [above right of=2] {};
\draw[-{Stealth[length=3mm, width=2mm]}] (1) -- (2);
\draw[-{Stealth[length=3mm, width=2mm]}] (3) -- (2);
\end{tikzpicture}
\hspace{2cm}
\begin{tikzpicture} [main/.style = {draw, circle}] 
\node[My Style] (1) {}; 
\node[My Style] (2) [below right of=1] {};
\node[My Style] (3) [above right of=2] {};
\draw[-{Stealth[length=3mm, width=2mm]}] (1) -- (2);
\draw[-{Stealth[length=3mm, width=2mm]}] (3) -- (2);
\draw[-{Stealth[length=3mm, width=2mm]}] (1) to [out=90,in=180,looseness=10] (1);
\draw[-{Stealth[length=3mm, width=2mm]}] (2) to [out=225,in=315,looseness=10] (2);
\draw[-{Stealth[length=3mm, width=2mm]}] (3) to [out=0,in=90,looseness=10] (3);
\end{tikzpicture} 
\caption{The digraphs $V$ (left), and $W$ (right).} \label{fig:M}  
\end{figure}

Now, it only remains to show that if $\varphi$ fixes $V$ and $W$, then it is the identity. 
Actually, we only need to check for the remaining 18 two-element classes of $\alpha \cap \beta$ (beyond the two containing $V$ and $W$) that they are fixed.
Let $\{H, H'\}$ be one of the 18 classes.
Note that we already know that $|u(H)|=|u(H')|$.
We develop some necessary conditions for $\varphi(H)=H'$.
In a similar fashion to $u(G)$ and $l(G)$, let $u(G_1, G_2)$ and $l(G_1,G_2)$ denote the set of common upper and lower covers of $G_1$ and $G_2$, respectively.  \\ 
{\it Condition 1A.} $|u(H,V)|=|u(H',V)|$. \\
{\it Condition 1B.} $|u(H,W)|=|u(H',W)|$. \\
1--1 of the 18 pairs fail Condition 1A--1B.
For those pairs, $\varphi(H)=H'$ is impossible, therefore, let us suppose now that the pair $\{H, H'\}$ passes the Conditions 1A and 1B.
Then, for some $k$, $u(H,V)=\{H_{V,1}, \dots, H_{V,k}\}$ and $u(H',V)=\{H'_{V,1}, \dots, H'_{V,k}\}$. 
Analogously, let $u(H,W)=\{H_{W,1}, \dots, H_{W,l}\}$ and $u(H',W)=\{H'_{W,1}, \dots, H'_{W,l}\}$.\\ 
{\it Condition 2A.} The two systems of integers 
$$|l(H_{V,1})|, \dots, |l(H_{V,k})| \;\;\text{ and }\;\; |l(H'_{V,1})|, \dots, |l(H'_{V,k})|$$
must be the same. (Note that, in systems, the order of the elements does not matter.)\\
{\it Condition 2B.} The same as Condition 2A, with $W$ instead of $V$, and $m$ instead of $k$.\\
2--2 of the remaining 16(=18-2) pairs fail Condition 2A--2B. 
Having 12 pairs satisfying all of our conditions so far, we are going to need some more. The conditions so far have been obviously sufficient conditions for $\varphi$ to be an automorphism, but the following two might need a bit explanation, which we'll give after stating them.  \\
{\it Condition 3A.} Let us suppose that the integer $|l(H_{V,i})|$ appears only once in the system $|l(H_{V,1})|, \dots, |l(H_{V,k})|$, and $|l(H_{V,i})|=|l(H'_{V,j})|$. 
Then for all $F$ outside our 18 pairs, we have
$$F \in l(H_{V,i}) \; \Longleftrightarrow \; F \in l(H'_{V,j}).$$
{\it Condition 3B.} The same as Condition 3A, with $W$ instead of $V$, and $m$ instead of $k$. \\
The reason why these two conditions are sufficient is that, at this point, we know that everybody on the third level, with possibly the exception of our 18 pairs, is fixed by $\varphi$ and, by its unique number of lower covers, we know that $\varphi(H_{V,i})=H'_{V,j}$, and hence a fixed lower cover of $H_{V,i}$ must be under its image as well. \\
Fortunately, all remaining pairs fail either of these two final conditions, in an even 6--6 distribution. 
Therefore, we are done. \qed

\section{Computer Calculations}

The computer calculations are performed in the software R (version 2022.02.3). 
The entire code is avalaible at the authors' websites:\\
\url{http://www.math.u-szeged.hu/~nfanni/poset.R} \\
\url{http://www.math.u-szeged.hu/~akunos/publ/poset.R} \\
Graphs are represented as matrices, i.e a graph of $n$ vertices is an $n\times n$ matrix, the vertices are numbered, and the $(i,j)$-th element is either 0 or 1. It is 1 iff there is a directed edge from the $i$-th vertex to the $j$-th. (A loop is a 1 in the corresponding diagonal element.)

In order to perform the calculations described in the previous sections, we need to construct the second, third and fourth levels of the poset. 
There are $2^{n^2}$ 0--1 matrices describing digraphs of the $n$-th level. 
However, some represent the same digraph, up to isomorphism. Hence, first, we need to identify the isomorphism types.  
The second level is simple to construct, even by hand.
There are $2^{2^2}=16$ matrices giving us 10 isomorphism types at the end.
For the third and fourth levels, calculating by hand is out of question, we need a computer.
We detail the steps for the third level here, calculating the fourth level is analogous.
Note that the procedure heavily relies on the previous (second) level.

For brevity, let us call two matrices {\it isomorphic} if the digraphs they represent are.
Before going through our method, let us note that, theoretically, all we would have to do is the following.
Create a list for all non-isomorphic matrices, then take each $3\times 3$ matrix one-by-one, check if it is isomorphic to any matrix previously added to this list, and if it is not, then add it to the list. 
This sounds simple, however, especially for the fourth level, it is very time-consuming to do.
Indeed, each time we try to add a matrix, we need to compare it, and all the ones isomorphic to it, to all the previous matrices. 
That is a lot of comparing which takes a long time. 
Therefore, we add a twist to our method, namely we also group the matrices by a property which we will need to use later anyway: their set of lower covers. 
In practice, we build two lists simultaneously: one, called $level3$, to contain the non-isomorphic matrices in smaller sublists, where, the digraphs represented by the matrices (within the sublists) share their lower covers. 
These sets of lower covers are stored in another list ($below3$) indexed in parallel with the previous list. 
So, for example, $level3[i]$ contains some matrices that share their lower covers, stored in $below3[i]$. 
All the matrices stored in $level3$ are pairwise non-isomorphic.
In the end, we can take $level3$, remove the barriers between its smaller lists, and get a list of all the non-isomorphic matrices of the third level. Our method is as follows:

\begin{enumerate}
    \item We go through all the possible $2^{3^2}=512$ matrices. Suppose that we are currently investigating the $n$-th matrix, $A$, and let $G_A$ denote the corresponding digraph.    
    \item There are at most six matrices that represent the same digraph as $A$---those we get by permuting the three labels of the vertices.
    We can get these matrices using the six permutation matrices, $P_1,\dots,P_6$ (one of them being the identity matrix). 
    We create the set of all matrices that represent the same isomorphism type as $A$: 
    $$I(A):=\{P_1 A P_1^{\top},\dots,P_6 A P_6^{\top}\}.$$
    \item
    We want to find the $2\times 2$ matrices of the list $level2$ that represent graphs which are covered by $G_A$. 
    We make a set, $below_A$, of all $level2$ matrices that are the $2\times2$-sized upper left corners of a matrix of $I(A)$.
    \item
    If $n=1$, meaning that $A$ is the first matrix investigated, then we create a list of matrices, $level3$, with the first element $level3[1]:=A$.
    Also, we define $below3[1]:=below_A$. 
    \item
    Let $n>1$. Suppose that the list $level3$ already has $N\in \mathbb{N}$ entries (all lists themselves). Then we already have a part of the third level constructed. We check if for some $i\leq N$ we have $below3[i]=below_A$. 
    
    If such an $i$ exists, then we have two possibilities. 
    Either there is an element of $I(A)$ that is already in $level3[i]$, in which case we discard $A$. 
    Otherwise, we add $A$ to $level3[i]$, since it shares its lower cover with the matrices already there.
    
    If there is no such $i$, meaning that no previously investigated matrix has the same lower covers as $A$, we add a new entry to $level3$, namely, $level3[N+1]:=A$, along with $below3[N+1]:=below_A$. 
\end{enumerate}
After $n=512$ we have the third level built up in the following way. 
We have a list of lists, $level3$, which contains exactly one matrix representation for every non-isomorphic graph having 3 nodes. 
The sublists contain those $3 \times 3$ matrices that have the same lower covers. The lower covers, $2 \times 2$ matrices, are stored in another list of lists with a corresponding structure.

We construct the fourth level similarly, there the lower covers are $3\times 3$ matrices. 
The fourth level is calculated in about 10 minutes, while the brute-force algorithm mentioned before would take days on the same computer.
Let us take a moment to understand why this method is that much faster than the brute-force one. 
The brute-force one starts relatively fast but the bigger part of the level already constructed, the slower it becomes. 
That is because it is very time-consuming to compare a $4\times4$ matrix and all its permutations to a long list of already stored matrices of the same size. 
Despite the fact that our method involves some comparisons of $4\times 4$ matrices as well, the trick is that the number of such comparisons is dramatically reduced. 
Technically, we use a single number to tag the matrices in the lists $level3$, $below3$ and their $
4$-dimensional counterparts.  
Therefore, when we have to find a matrix or a list of matrices in a long list, we are actually just matching simple numbers instead of matrices. 
Sometimes it is necessary to see the matrix form itself, but we can bring it up using its label nonetheless.


Now we are done with the harder part of the problem since we have built a structure that can easily be handled.
What remains to be done is to to find the pairs of $\alpha$ and $\beta$ and then check conditions 1A-3B, which is computationally pretty straightforward. 
The results we obtain were already delineated in the previous section.

\section{Future Prospects}

To prove the conjecture, at least in the fashion of the present paper, one should calculate $\aut_{12}\mathcal{D}_{13}$.
This we expect to be isomorphic to $\mathcal{C}$, and if verified, this would yield the conjecture with the use of Proposition \ref{43426869}.
The problem is that calculating even $\aut_{3}\mathcal{D}_{4}$ turned out to be quite compute-intensive, even though we didn't do it brute-force at all. 
The determination of $\aut_{4}\mathcal{D}_{5}$ might be in reach using some clever programming or preexisting programs (instead of our approach of building everything from scratch).
But that would still be very far from the yearned $\aut_{12}\mathcal{D}_{13}$. 

Another possible approach would be to try to reduce the number 12 in Proposition \ref{43426869}.
If the conjecture investigated is true, then 12 can be reduced to 3. 
One might try to have a look at the proof in  \cite{Ksubstructure} and attempt to reduce 12 in it.
We are sure this can be done, but it looks tedious.

To sum up, proving the conjecture does not seem trivial at this point, but in mathematics some bright new idea can always come and tame an ostensibly challenging problem.

\bibliographystyle{abbrv}
\bibliography{references}

\end{document}